\documentclass[12pt,twoside,a4paper]{article}

\usepackage{amsmath}
\usepackage{a4wide}
\usepackage{graphicx}
\usepackage{dsfont}
\usepackage{amsthm}
\usepackage{amssymb}

\newtheorem{theorem}{Theorem}
\newtheorem{definition}{Definition}
\newtheorem{example}{Example}

\newtheorem{remark}{Remark}

\newtheorem{proposition}{Proposition}
\newtheorem{algorithm}{Algorithm}

\title{Geometry of rational helices and its applications}
\author{Fatma \c{S}eng\"{u}ler-\c{C}ift\c{c}i \\[2ex]
Johann Bernoulli Institute for Mathematics and Computer Science\\ University of Groningen\\ PO Box 407\\ 
9700 AK Groningen, The Netherlands \\
F.A.Senguler-Ciftci@rug.nl}

\begin{document}
\maketitle

\abstract{
The present paper attempts to show an alternative approach 
with regards to rational
Pythagorean-hodograph (PH) curves and especially more natural approach for rational PH helices (i.e. rational helices).
It exploits geometric features
of rational helices to obtain a simpler construction of these curves
and apply this to related subjects. 
One of these applications is \emph{Geometric $C^1$ Hermite interpolation}
(i.e. interpolation of end points
with associated unit tangents) by rational helices. Furthermore, we investigate the existence of 
rational rotation minimizing frames (RRMFs) on rational helices. A rational approximation
procedure to rotation minimizing frames (RMFs) is suggested.
Subsequently, we deploy the approximate frame for modeling a rational sweep surface.
The resulting algorithms are illustrated by several examples.  
\vspace*{1cm}


\noindent 

\section{Introduction}
\label{sec:Introduction}
A parametric \emph{rational} curve $\mathbf{r}(t)$ is called a \emph{Pythagorean-hodograph} (PH) curve, if its
parametric speed $\sigma(t) = |\mathbf{r}'(t)|$ is a rational polynomial of the time parameter \cite{Faroukietal98, Farouki08}. 
More explicitly, these curves are defined as \emph{rational}  curves $\mathbf{r}(t)=(x(t),y(t),z(t))$ fulfilling the distinguishing condition 
\begin{equation*}
\mathbf{r}'(t) \cdot \mathbf{r}'(t) = x{'}^2(t) + y{'}^2(t) + z{'}^2(t) = \sigma^2(t),
 \end{equation*}
  for some rational polynomial $\sigma(t)$.
The theory of rational PH curves is an extensively studied
research topic in Computer Aided Geometric Design (CAGD), because of their distinguishing features from \emph{ordinary} polynomial curves. 
Their parametric speed, curvatures and offsets depend 
\emph{rationally} on the curve parameter. Furthermore, only the PH curves admit rational unit tangents or more 
generally rational adapted frames \cite{WagnerRavani97}, and 
they have \emph{exact} rotation minimizing frames (RMFs) \cite{Farouki02}.
 
The theory of polynomial PH curves is quite well-established, but 
rational PH curves have not been investigated enough except the planar rational PH curves case \cite{Farouki08, Fiorot94, Pottmann95a}. 
The extension from polynomial to rational PH curves is non-trivial, because a rational hodograph does not always yield a rational curve upon 
integration. A comparison of polynomial and rational planar PH curves can be found in \cite{Farouki08} and \cite{Pottmann95a} - see also {\v S}ir et al \cite{Sir2010}. 
Farouki and  {\v S}ir \cite{FaroukiSir11} recently have given a formulation for rational PH space curves
by using the Euclidean geometry of space curves. They identified a space curve as the edge of
regression of its tangent developable surface. This consideration led them to express a
rational PH curve as follows. 
\begin{proposition}
Let $\mathbf{v}(t)$ be any rational vector field and
$\mathbf{r}(t)$ be a rational space curve. Then $\mathbf{v}(t)$
is tangent to $\mathbf{r}(t)$ if and only if
there exists a rational function $f(t)$ such that 
\begin{equation*}
\mathbf{r}(t)= \,\frac{f(t)\,\mathbf{w}'(t)\times \mathbf{w}''(t)+f'(t)\,\mathbf{w}''(t) \times \mathbf{w}(t)+f''(t)\,\mathbf{w}(t) \times \mathbf{w}'(t)}
{\mathbf{w}(t)\cdot (\mathbf{w}'(t) \times \mathbf{w}''(t))},
\end{equation*}
where $\mathbf{w}(t)=\mathbf{v}(t) \times \mathbf{v}'(t)$.
\end{proposition}
\noindent
The interpretation of Proposition 1 is a characterization of rational curve with rational tangent field. Rational PH curves $\mathbf{r}(t)$ 
are obtained from Proposition 1 when the tangent field has rational norm (or is a rational unit vector field for example).
 It means the unit tangent $\mathbf{t}(t)$ of a PH curve $\mathbf{r}(t)$ is has a rational dependence on the parameter $t.$

See that in the proposition above it has the relation
\begin{equation}\label{eq:equation_function_f}
f(t)= \det(\mathbf{r}(t),\mathbf{v}(t),\mathbf{v}'(t)).
\end{equation}
They considered Hermite interpolation problem in detail by using rational PH curves.
Additionally, many relevant directions are pointed out in
\cite{FaroukiSir11} which stimulated our work.

Among others \emph{rational helices} form a special class
of rational PH curves because of their rich geometric aspects, such as having rational
Frenet-Serret frames.
These curves are curves with constant curvature and torsion ratios: i.e., $\frac{\tau}{\kappa}=\mbox{constant}$, where
\begin{equation}\label{eq:torsion}
\kappa=\frac{\mathbf{r}'(t) \times \mathbf{r}''(t)}{|\mathbf{r}'(t)|^3} \quad \text{and} \quad \tau=\frac{(\mathbf{r}'(t)\times\mathbf{r}''(t))\cdot \mathbf{r}'''(t)}{|\mathbf{r}'(t) \times\mathbf{r}''(t)|^2}.
\end{equation}
If a rational curve $\mathbf{r}(t)$ is a helix, then its hodograph $\mathbf{r}'(t)$
has to be Pythagorean. 
(Therefore, we shortened the name `rational PH helices' with `rational helices' along the paper.)
The converse of this statement is not generally true: 
Although all PH cubics are helices, there exist PH quintics which are not helical, for instance \cite{Farouki08}, \cite{BeltranMonterde07}, \cite{Monterde09}. 
Due to their importance in applications rational helices deserve special attention, which is the subject of this paper.

In the present paper, we firstly give a construction of a rational PH curve
in a slightly different way which requires less differential geometric background than Farouki and {\v S}ir \cite{FaroukiSir11}.
Our construction is similar to what is done in line geometry to
obtain the striction curve of a tangent developable surface \cite{PottmannWallner01}.
Moreover, the new approach is more adapted to rational helices. 
One of the distinctive features of these curves
is that their unit tangent traces a small circle on the sphere  \cite{FaroukiSir11}, \cite{Han11}, \cite{doCarmo76}.
This basis is constructed by using the property that a helix makes a constant angle with
a fixed direction \cite{doCarmo76}.  If the \emph{unit tangent} of the rational helix is
$\mathbf{t}(t)$ and a \emph{unit vector} which makes constant angle with $\mathbf{t}(t)$
is $\mathbf{u}$, then a rational orthonormal frame is constructed by
a Gram-Schmidt process. Furthermore, it is shown that if the unit tangent and one of the rational function are given,
 we can construct the rational helix by a simple algorithm.

An \emph{adapted frame} on a space curve  $\mathbf{r}(t)$ is an orthonormal moving frame 
$(\mathbf{f}_1(t),\mathbf{f}_2(t),$ $\mathbf{f}_3(t))$ 
such that, $\mathbf{f}_1(t)=\mathbf{t}(t)$ is the unit tangent $\mathbf{r}'(t)/|\mathbf{r}'(t)|,$ 
and the other two vectors span the normal plane. 
A \emph{rotation minimizing frame} (RMF)
 $(\mathbf{t}(t),\mathbf{f}_2(t),\mathbf{f}_3(t))$ of a curve in space
consists of the tangent $\mathbf{t}(t)$ and two
normal vectors $\mathbf{f}_2(t)$ and $\mathbf{f}_3(t)$ which rotate as little as possible around 
$\mathbf{t}(t)$ that makes them distinguished among adapted frames. RMF has been studied by several authors
\cite{Farouki02, BartonJuttlerWang10, fatma-gert-2013, fatma-gert-2013eurocg,
  pottmann1997principal, pottmann1998contributions, Faroukietal12a, FaroukiSakkalis10}. These frames are used in animation, robotics applications,
the construction of swept surfaces \cite{Klok86} where 
the axis of a tool should remain tangential to a given spatial path while minimizing changes 
of orientation about this axis. 
From the point of view of applications rational helices are
 important to posses an associated rational approximation to RMFs as we explain below.
RMFs on rational helices are not rational in general, therefore our approximation
is hoped to be of great relevance.

We develop a rational approximation to RMFs on rational helices with a high accuracy method. For doing this,
we intend to extend our earlier work \cite{fatma-gert-2013eurocg} to rational helices. Besides that we invoke rational
 approximation methods which appear in the literature, e.g. \cite{FaroukiHan03, Maurer99}.
An immediate application of this rational approximation is rational approximation to
sweep surfaces. A \emph{sweep surface} is defined by sweeping of a plane curve $c(s)=(c_1(s),c_2(s))^T,$  (\emph{profile curve}) along a given $\mathbf{r}(t)$
(\emph{spine curve}), i.e,
\begin{equation}
\textbf{S}(s,t)=\,\mathbf{r}(t)+c_1(s)\,\mathbf{f}_2(t)+ c_2(s)\,\mathbf{f}_3(t),
\end{equation}
where $(\mathbf{f}_1(t),\mathbf{f}_2(t),\mathbf{f}_3(t))$ is an adapted
frame along $\mathbf{r}(t)$. This is a useful method for generating surfaces in computer graphics and geometric modeling \cite{Wang97}.  
As proposed by Pottmann and Wagner \cite{pottmann1997principal, pottmann1998contributions} RMF generated sweeping surfaces, which are called \emph{profile surfaces}, 
are of grate importance in surface modeling \cite{WagnerJuttler99}. 
Regrettably, profile surfaces with rational spine curve and cross section are generally not rational, even if the spine is a PH curve. 
Here, we focus on surfaces with rational helix spines and cross section curves. 

Our plan for this paper is as follows. In Section \ref{sec:Rational_PH_curves} we discuss different ways of constructing rational PH curves,
 while Section \ref{sec:Rational_PH_helices} is devoted to rational helices. In Section \ref{Hermite} we show how the ideas of classical 
$C^1$ Hermite interpolation can be utilized to rational parameterizied helices. In Section \ref{sec:RRMF} we give a necessary and sufficient
 condition for rational helices of any degree to have rational rotation minimizing frames (RRMFs), which is followed by a nonexistence result for curves
with unit tangent of degree $2$. In Section \ref{sec:rational_approximation_RMF}
an approximation method to RMFs is given, while  this method is applied 
in Section \ref{sec:surface} to the problem of rational approximations to
profile surfaces. Finally, in Section \ref{concl} we conclude with some remarks about our future considerations.

\section{Another construction of rational PH curves}
\label{sec:Rational_PH_curves}

In this section, we derive rational PH curves with a new insight.
Afterwards, we suggest a simpler derivation for rational helices in the next section.
We first recall the definition of a rational PH curve \cite{Farouki08} for completeness.

\begin{definition}[Rational PH curve]
Let $\mathbf{r}(t)$ be a rational space curve. If its speed is a rational function, i.e.,
$|\mathbf{r}'(t)|=\sigma(t)$ for some rational function $\sigma(t)$, then $\mathbf{r}(t)$
is called a rational PH space curve.
\end{definition}

Now, we give a direct way of constructing a rational PH curve $\mathbf{r}(t)$ inspired by
the results in \cite{FaroukiSir11}. Let a
rational vector field $\mathbf{v}(t)$ and a rational
function $g(t)$ be given.
Assume that $\mathbf{v}(t)$ is tangent to $\mathbf{r}(t)$:
\begin{equation}\label{eq:dr}
 \mathbf{r}'(t)=g(t)\, \mathbf{v}(t),
\end{equation}
for some rational function $g(t)$.
Throughout the paper, we assume that $\mathbf{v}(t)$ and $\mathbf{v}'(t)$
are linearly independent for all $t$. Subsequently, we express
$\mathbf{r}(t)$ in terms of the basis $\left(\mathbf{v}(t), \mathbf{v}'(t), \mathbf{v}(t) \times \mathbf{v}'(t) \right)$
of $\mathbb{R}^3$. For this purpose, we set the rational PH curve as 
\begin{equation}\label{eq:rational_PH_curve}
 \mathbf{r}(t)=a_1(t)\,\mathbf{v}(t)+a_2(t)\,\mathbf{v}'(t)+a_3(t)\,\mathbf{v}(t) \times \mathbf{v}'(t).
\end{equation}
We assert that one can eliminate $a_1(t)$ and $a_2(t)$ from \eqref{eq:rational_PH_curve} by the PH curve property \eqref{eq:dr}.
To show this, we use the following three identities which we obtain from inner products of the basis elements
 $(\mathbf{v}(t),\mathbf{v}'(t),\mathbf{v}(t) \times \mathbf{v}'(t))$
by both sides of \eqref{eq:dr}:
\begin{equation}\label{eq:first_orthogonality}
 \mathbf{r}'(t) \cdot \mathbf{v}(t)=g(t)\,\mathbf{v}(t) \cdot \mathbf{v}(t), 
\end{equation}
\begin{equation}\label{eq:second_orthogonality}
 \mathbf{r}'(t) \cdot \mathbf{v}'(t)=g(t) \,\mathbf{v}(t) \cdot \mathbf{v}'(t),
\end{equation}
and
\begin{equation}\label{eq:third_orthogonality}
 \mathbf{r}'(t) \cdot (\mathbf{v}(t) \times \mathbf{v}'(t))=0.
\end{equation}
The right hand side of the equation \eqref{eq:third_orthogonality} is zero, since the three vectors $ \mathbf{v}(t), \mathbf{v}'(t), \mathbf{v}(t) \times \mathbf{v}'(t)$ (each supposedly different from zero)  are coplanar by the equation \eqref{eq:dr}.  
We first take the derivative of \eqref{eq:rational_PH_curve} to obtain
\begin{equation}\label{eq:rational_PH_curve_derivative}
 \mathbf{r}'(t)=a_1'(t)\mathbf{v}(t)+(a_1(t)+a_2'(t))\mathbf{v}'(t)+a_2(t)\mathbf{v}''(t)+a_3'(t)\mathbf{v}(t) \times \mathbf{v}'(t)+
a_3(t)\mathbf{v}(t) \times \mathbf{v}''(t).
\end{equation}
So, by the equations \eqref{eq:first_orthogonality} and \eqref{eq:rational_PH_curve_derivative}, one has 
\begin{equation}\label{eq:first_relation}
 (a_1'(t)-g(t)) \,\mathbf{v}(t) \cdot \mathbf{v}(t)+(a_1(t)+a_2'(t))\mathbf{v}(t) \cdot \mathbf{v}'(t)+
a_2(t)\,\mathbf{v} (t)\cdot \mathbf{v}''(t)=0.
\end{equation}
Similarly, by equations \eqref{eq:second_orthogonality} and \eqref{eq:rational_PH_curve_derivative}, it is obtained that
\begin{align}\label{eq:second_relation}
& (a_1'(t)-g(t)) \,\mathbf{v}(t) \cdot \mathbf{v}'(t)+(a_1(t)+a_2'(t))\,\mathbf{v}'(t) \cdot \mathbf{v}'(t)+
a_2(t)\,\mathbf{v}'(t) \cdot \mathbf{v}''(t)\\[1ex]\nonumber
-&a_3(t)\,\det(\mathbf{v}(t),\mathbf{v}'(t),\mathbf{v}''(t))=0,
\end{align}
where we use the following property of the triple scalar product:
\begin{equation}
\label{eq:detv}
(\mathbf{v}(t)\times \mathbf{v}'(t))\cdot \mathbf{v}''(t)= \det(\mathbf{v}(t),\mathbf{v}'(t),\mathbf{v}''(t)).
\end{equation}
One can observe that the equations \eqref{eq:first_relation}
and \eqref{eq:second_relation} enable us to eliminate $a_1'(t)$ and $g(t)$ to have 
\begin{align}\label{eq:first+second_relation}
 (&a_1(t)+a_2'(t))\left| \begin{array}{cc}
  \mathbf{v}(t) \cdot \mathbf{v}'(t)  & \mathbf{v}'(t) \cdot \mathbf{v}'(t)\\
 \mathbf{v}(t) \cdot \mathbf{v}(t)  &  \mathbf{v}(t) \cdot \mathbf{v}'(t) \\
\end{array} \right| 
 +
a_2(t)\left| \begin{array}{cc}
  \mathbf{v}(t) \cdot \mathbf{v}'(t)  &  \mathbf{v}'(t) \cdot \mathbf{v}''(t) \\
\mathbf{v}(t)\cdot \mathbf{v}(t) &  \mathbf{v}(t) \cdot \mathbf{v}''(t) \\
\end{array} \right|\\[1ex]\nonumber
&+a_3(t)(\mathbf{v}(t) \cdot \mathbf{v}(t))\det(\mathbf{v}(t),\mathbf{v}'(t),\mathbf{v}''(t))=0.
\end{align}
Also, by equations \eqref{eq:third_orthogonality} and \eqref{eq:rational_PH_curve_derivative} we have that
\begin{equation}\label{eq:third_relation}
a_2(t)\,\det(\mathbf{v}(t),\mathbf{v}'(t),\mathbf{v}''(t))+a_3'(t)\,(\mathbf{v}(t) \times \mathbf{v}'(t))^2+a_3(t)\,(\mathbf{v}(t) \times \mathbf{v}'(t))\cdot
(\mathbf{v}(t) \times \mathbf{v}''(t))=0.
\end{equation}
Note that, here one can further compute
\begin{equation*}
(\mathbf{v}(t)\times\mathbf{v}'(t))^2= (\mathbf{v}(t) \cdot \mathbf{v}(t))\, (\mathbf{v}'(t) \cdot \mathbf{v}'(t)) - (\mathbf{v} (t)\cdot \mathbf{v}'(t))^2,
\end{equation*}
and
\begin{equation*}
(\mathbf{v}(t)\times\mathbf{v}''(t)) \cdot  (\mathbf{v}(t)\times\mathbf{v}'(t))=( \mathbf{v}(t) \cdot \mathbf{v}(t))\,(\mathbf{v}''(t) \cdot \mathbf{v}'(t)) - (\mathbf{v}(t)\cdot\mathbf{v}'(t))\,(\mathbf{v}(t) \cdot\mathbf{v}''(t)).
\end{equation*}

If a rational vector field $\mathbf{v}(t)$ and a rational function $a_3(t)$ are given,
then by \eqref{eq:third_relation} the function $a_2(t)$ can be determined, and being
determined $a_2(t).$ We can obtain $a_1(t)$ by utilizing \eqref{eq:first+second_relation}.
Hence, all elements of rational PH curve given in the form \eqref{eq:rational_PH_curve} are found.
Therefore, we can state the following proposition.
\begin{proposition}
Let $\mathbf{v}(t)$ be any rational vector field and
$\mathbf{r}(t)$ be a rational space curve. Then $\mathbf{v}(t)$
is tangent to $\mathbf{r}(t)$ if and only if
there exists a rational function $a_3(t)$ such that 
$\mathbf{r}(t)$ is given as in \eqref{eq:rational_PH_curve},
where $a_1(t)$ and  $a_2(t)$ are obtained by \eqref{eq:first+second_relation}
and  \eqref{eq:third_relation}, respectively.
\end{proposition}

\begin{proof}
Let us assume that $\mathbf{v}(t)$
is tangent to $\mathbf{r}(t)$, then from  the argument above, $\mathbf{r}(t)$ is evoked in the desired form.
Conversely, let $\mathbf{r}(t)$ be given as in \eqref{eq:rational_PH_curve},
where $a_1(t)$ and  $a_2(t)$ are obtained by \eqref{eq:first+second_relation}
and  \eqref{eq:third_relation}, respectively. Then by taking the derivative of 
\eqref{eq:rational_PH_curve} and using  \eqref{eq:first+second_relation}
and  \eqref{eq:third_relation},  one can easily check that $\mathbf{v}(t)$
is tangent to $\mathbf{r}(t)$.
\end{proof}

If the rational field $v(t)$ has rational norm, then 
he following algorithm describes the construction of a rational PH curve. 
 \begin{algorithm}
   \label{algo:RPHcurve}  \textsc{RationalPHcurve}$(a_3(t),\mathbf{v}(t))$ \\[1ex]
  Input: Scalar function $a_3(t)$ and rational tangent vector $\mathbf{v}(t).$ \\[1ex]
    1. Construct the basis $\left(\mathbf{v}(t), \mathbf{v}'(t), \mathbf{v}(t) \times \mathbf{v}'(t) \right)$.\\[1ex]
    2. Apply \eqref{eq:third_relation} to get $a_2(t).$ \\[1ex]
    3. Apply \eqref{eq:first+second_relation} to get $a_1(t).$ \\  [1ex]         
 Output:   Rational PH curve \eqref{eq:rational_PH_curve}.
 \end{algorithm}
 
One can conclude by \eqref{eq:rational_PH_curve} that 
\begin{equation}
\label{eq:a3PHcurves}
 a_3(t)=\frac{\det(\mathbf{r}(t),\mathbf{v}(t),\mathbf{v}'(t))}{(\mathbf{v}(t) \times \mathbf{v}'(t))^2},
\end{equation}
which gives
\begin{equation}
 a_3(t)=\frac{f(t)}{(\mathbf{v}(t) \times \mathbf{v}'(t))^2},
\end{equation}
by comparison to equation \eqref{eq:equation_function_f}. This equation compares two 
approaches for obtaining rational PH curves. 

It is also worth to mention here that, if $\mathbf{v}(t)$ is unit,
then all the formulae above simplifies considerably. In that case 
the basis $(\mathbf{v}(t),\mathbf{v}'(t),\mathbf{v}''(t))$ becomes orthogonal
(not necessarily orthonormal),
since $|\mathbf{v}(t)|=1$ implies that $\mathbf{v}(t) \cdot \mathbf{v}'(t)=0$.

\section{Rational helices}
\label{sec:Rational_PH_helices}
The curvature and the torsion of a helical space curve are constant. Mainly for that reason
these curves are important. Additionally, rationality assumption add more importance for them. So we analyze rational helices in that section with that motivation. 
Note thet, the identification of helices with curves that have small circle tangent 
 indicatrices on the unit sphere, used in this section, has been exploited 
 by many authors \cite{BeltranMonterde07}, \cite{Monterde09},  \cite{Han11}, \cite{Faroukietal04}, \cite{Faroukietal09c}, \cite{Faroukietal09b}, \cite{Han10}.

From the derivation in the previous section, it is seen that
a rational curve can be obtained by a long computation.
In this section, we exploit
geometric features of rational helices to simplify the formula which gives rational PH curves.

Let $\mathbf{r}(t)$ be a rational helix such that
\begin{equation}\label{eq:PH_helix_definition}
 \mathbf{r}'(t)=\sigma(t)\, \mathbf{t}(t),
\end{equation}
where $\mathbf{t}(t)$ is the unit tangent of $\mathbf{r}(t)$
which is a rational vector field
and $\sigma(t)$ is the scalar speed which is a rational function.
Then, by the definition of a \emph{helix}  \cite{doCarmo76}, $\mathbf{t}(t)$ makes
a \emph{constant} angle $\psi$ with
a unit \emph{constant} vector field $\mathbf{u}$:
\begin{equation}
\label{eq:helixCon}
 \mathbf{t}(t) \cdot \mathbf{u}=\cos\psi \quad \mbox{for all} \,\, t.
\end{equation}
It is known that \cite{Farouki08}, from equation \eqref{eq:helixCon} it can be concluded that every rational helix is a PH curve.
In detail, since $\mathbf{t}(t)=\frac{\mathbf{r}'(t)}{|\mathbf{r}'(t)|}$ and by \eqref{eq:helixCon}, one obtains
\begin{equation*}
 \mathbf{r}'(t) \cdot \mathbf{u}=\cos\psi \,|\mathbf{r}'(t)|  \quad \mbox{for all} \,\, t,
\end{equation*}
which shows that $|\mathbf{r}'(t)|$ is rational.
 
\begin{figure}
\begin{center}
\includegraphics[angle=0,width=7cm]{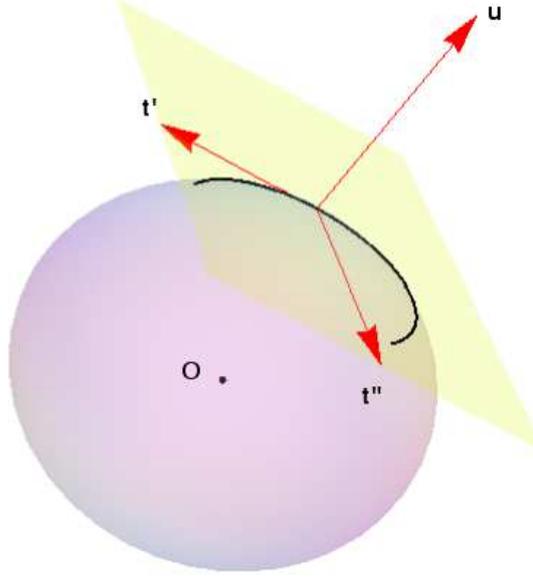}
\end{center}
\caption{\label{fig:vector_u}
Construction of the unit vector $\mathbf{u}$.
}
\end{figure}
When the unit tangent $\mathbf{t}(t)$ is given, it is a straitforward to obtain the unit vector $\mathbf{u}$;
it is the same direction with $\mathbf{t}'(t) \times \mathbf{t}''(t)$ (see Figure \ref{fig:vector_u}).
We also point out here that $\mathbf{t}(t)$ traces a small circle on the unit sphere
as $\mathbf{r}(t)$ is a helix \cite{doCarmo76}, while great circles give rise to
planar curves. 

We consider the rational orthonormal
basis $(\mathbf{w}_1(t),\mathbf{w}_2(t),\mathbf{w}_3(t))$
of $\mathbb{R}^3$, where
\begin{equation}\label{eq:helix_frame}
\begin{split}
 \mathbf{w}_1(t)&=\mathbf{t}(t), \\
\mathbf{w}_2(t)&=\frac{1}{\sin\psi}\,\mathbf{u} \times \mathbf{t}(t) \\
\mathbf{w}_3(t)&=\mathbf{w}_1(t)\times \mathbf{w}_2(t). \\
\end{split}
\end{equation}
Here, from the vector product of two vectors $\mathbf{w}_1(t)$ and $\mathbf{w}_2(t)$ observe that
\begin{equation}\label{equation_w_3}
\mathbf{w}_3(t)=\frac{1}{\sin\psi}\,\left(\mathbf{u}-\cos\psi\, \mathbf{t}(t)\right). 
\end{equation}
Detect also that, this frame is aligned with the Frenet-Serret Frame (FSF) up to suitable
orientation. Additionally, the main advantage in the use of this basis
is its simple nature, for instance its construction does not rely on
derivatives.

We express the rational helix $\mathbf{r}(t)$ according to this basis by
\begin{equation}\label{eq:rational_PH_helix}
 \mathbf{r}(t)=a_1(t)\,\mathbf{t}(t)+a_2(t)\,\mathbf{w}_2(t)
+a_3(t)\mathbf{w}_3(t).
\end{equation}
Taking the derivative of the rational helix $\mathbf{r}(t)$ yields in
\begin{equation}\label{eq:rational_PH_helix_derivative}
 \mathbf{r}'(t)=a_1'(t)\mathbf{t}(t)+a_1(t)\mathbf{t}'(t)+
a_2'(t)\mathbf{w}_2(t)+a_2(t)\mathbf{w}_2'(t)+a_3'(t)\mathbf{w}_3(t)+
a_3(t) \mathbf{w}_3'(t).
\end{equation} 
We are taking the following two equations into account: 
\begin{equation}\label{eq:first_orthogonality_helix}
 \mathbf{r}'(t) \cdot \mathbf{w}_2(t)=0, 
\end{equation}
\begin{equation}\label{eq:second_orthogonality_helix}
 \mathbf{r}'(t) \cdot \mathbf{w}_3(t)=0.
\end{equation}
So, by \eqref{eq:rational_PH_helix_derivative} and 
\eqref{eq:first_orthogonality_helix} we have
\begin{equation*}
 a_1(t)\, \mathbf{t}'(t) \cdot \mathbf{w}_2(t) +a_2'(t) + a_3(t)\,\mathbf{w}_3'(t) \cdot \mathbf{w}_2(t)=0, 
\end{equation*}
or equivalently, we have
\begin{equation}\label{eq:first_orthogonality_result_helix}
 a_2'(t) +(a_1(t)-\cot\psi \,a_3(t))\csc\psi\,\det(\mathbf{u},\mathbf{t}(t),\mathbf{t}'(t))=0.
\end{equation}
Accordingly, by equations \eqref{eq:rational_PH_helix_derivative} and \eqref{eq:second_orthogonality_helix}, we have
\begin{equation*}
a_1(t)\,\mathbf{t}'(t) \cdot \mathbf{w}_3(t)
+ a_2(t)\,\mathbf{w}_2'(t) \cdot \mathbf{w}_3(t) + a_3'(t)=0, 
\end{equation*}
or equivalently, we have
\begin{equation}\label{eq:second_orthogonality_result_helix}
 a_2(t)\,\cot\psi\,\csc\psi\,\det(\mathbf{u},\mathbf{t}(t),\mathbf{t}'(t)) +a_3'(t)=0.
\end{equation}
Furthermore, by comparison of \eqref{eq:PH_helix_definition}
and \eqref{eq:rational_PH_helix_derivative} we have 
\begin{equation*}
 \sigma(t)=a_1'(t)-\frac{a_2(t)}{\sin \psi}\, \det(u,\mathbf{t}(t),\mathbf{t}'(t)).
\end{equation*}

Therefore, we can conclude that given a rational
vector field $\mathbf{t}(t)$ which traces a
small circle on the unit sphere and a rational function
$a_3(t)$, a rational helix can be constructed. 
Then one can obtain $a_2(t)$ by \eqref{eq:second_orthogonality_result_helix}, and
finally $a_1(t)$ can be derived by \eqref{eq:first_orthogonality_result_helix}.

We are in a position to state the following proposition whose proof is omitted as being straightforward.
\begin{proposition}
Let $\mathbf{t}(t)$ be rational unit vector field which traces
a small circle on the unit sphere $S^2$ and
$\mathbf{r}(t)$ be a rational helix. Then $\mathbf{t}(t)$
is tangent to $\mathbf{r}(t)$ if and only if
there exists a rational function $a_3(t)$ such that 
$\mathbf{r}(t)$ is given as in \eqref{eq:rational_PH_helix},
where $a_2(t)$ and $a_1(t)$ are obtained by \eqref{eq:first_orthogonality_result_helix} 
and \eqref{eq:second_orthogonality_result_helix}, respectively.
\end{proposition}

The following algorithm describes the construction of a rational helix. 

    \begin{algorithm}\textsc{RationalHelix}$(a_3(t),\mathbf{r}(t))$
     \label{algo:RPHhelix}\\[1ex]
    Input: Scalar function $a_3(t)$ and unit tangent vector $\mathbf{t}(t).$ \\[1ex]
    1. Find the constant direction of rational helix \\[1ex]
   \begin{equation}
 \mathbf{u}  =\frac{\mathbf{t}'(t)\times\mathbf{t}''(t)}{\|\mathbf{t}'(t)\times\mathbf{t}''(t)\|}.
   \end{equation}
        2.  Apply  \eqref{eq:helixCon} to get the constant angle $\psi.$  \\  [1ex]       
        3. Construct the basis $(\mathbf{w}_1(t),\mathbf{w}_2(t),\mathbf{w}_3(t))$ from equations \eqref{eq:helix_frame}.\\[1ex]
        4. Find $a_2(t)$ from equation \eqref{eq:second_orthogonality_result_helix} and secondly apply  \eqref{eq:first_orthogonality_result_helix} to get $a_1(t).$\\[1ex]
    Output: Rational helix curve \eqref{eq:rational_PH_helix}. 
    \end{algorithm}
    
\begin{example}
\label{ex:cubic}
In this example $\mathbf{t}(t)$ is chosen to be
\begin{equation}\label{eq:Steographic_projection}
 \mathbf{t}(t)=\frac{(2\,b_1(t),2\,b_2(t),b_1^2(t)+b_2^2(t)-1)}{b_1^2(t)+b_2^2(t)+1},
\end{equation}
where $b_1(t)=-\frac{1}{2} t+1$ and  $b_2(t)= 2t-1 $. As, $b_1(t)$ and $b_2(t)$
are linear $\mathbf{t}(t)$ traces a circle on the unit sphere (see Fig. \ref{fig:example_1}).
The rational function $a_3(t)$ is chosen as a rational Bezier curve of degree $3$ given by
\begin{equation}\label{eq:rational_Bezier}
 a_3(t)=\frac{c_0\, B_0^3(t)+c_1 w_1 \,B_1^3(t)+c_2 w_2 \,B_2^3(t)+c_3\, B_3^3(t)}
{B_0^3(t)+w_1\, B_1^3(t)+w_2\, B_2^3(t)+B_3^3(t)},
\end{equation}
where  $B_i^3(t)$ the ith Bernstein polynomial of degree 3:
\begin{equation*}
 B_i^3(t)={{3}\choose{i}}\, t^{i} (1-t)^{3-i} , \, i=1,2,3.
\end{equation*}
Especially we take 
\begin{equation*}
 c_0 = 1, \, c_1 = 2, \, c_2 = -3, \, c_3 = \frac{1}{2}, \, w_1 = 3, \, w_2 = 1.
\end{equation*}
Hence, we obtain a rational helix of degree $9$ (see Fig. \ref{fig:example_1}).  
\begin{figure}
\begin{center}
\includegraphics[angle=0,width=6cm]{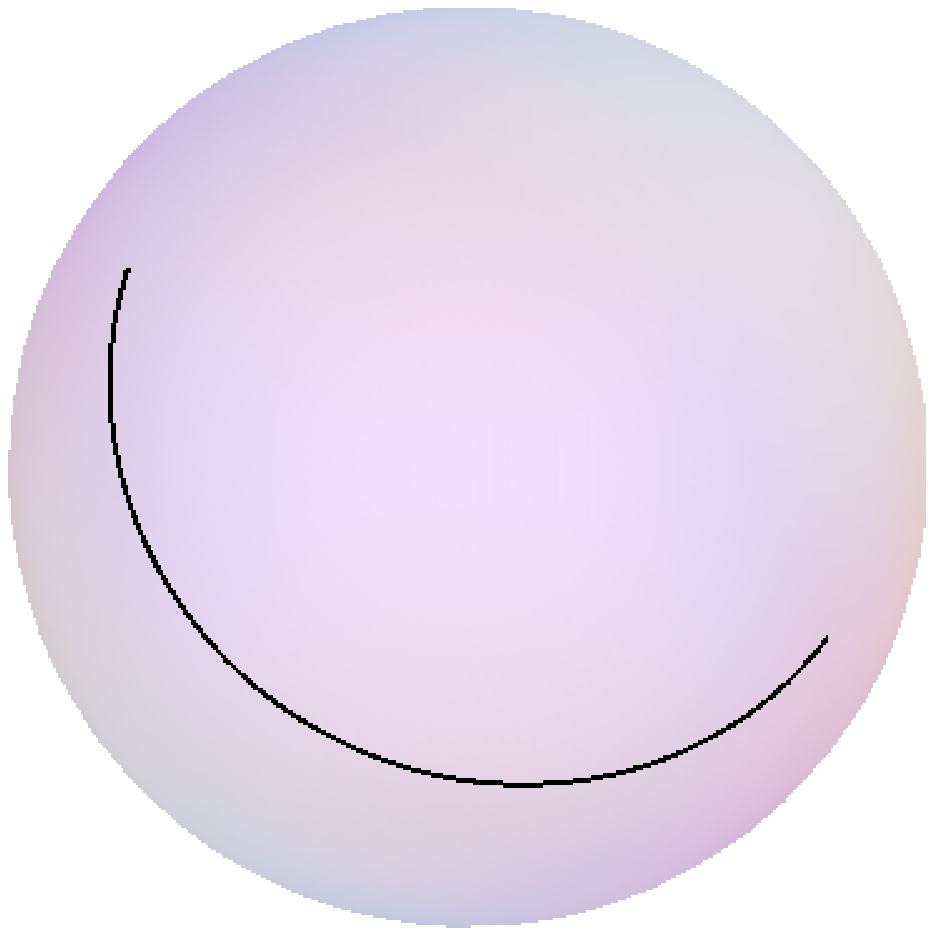}
\includegraphics[angle=0,width=6cm]{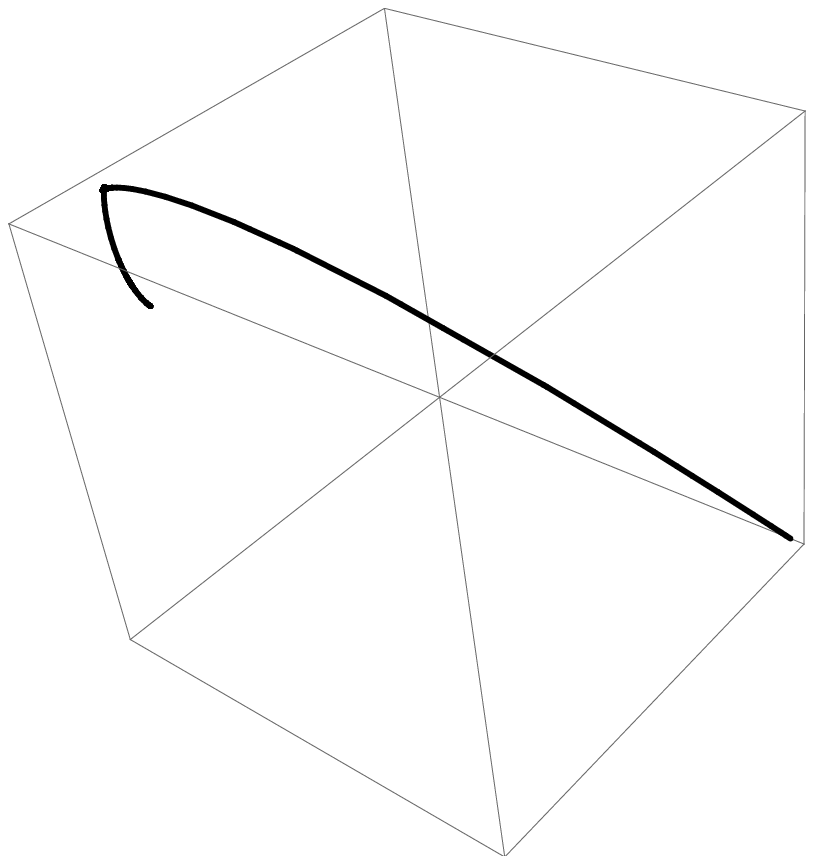}
\end{center}
\caption{\label{fig:example_1} Left: Tangent indicatrix of the rational helix in Example
\ref{ex:cubic}. Right: The rational helix of degree $9$ ($0<t<1$).
}
\end{figure}
\end{example}

\begin{example}
\label{ex:PHcubic}
We choose in \eqref{eq:Steographic_projection} that $b_1(t)=-3\, t+1$ and  $b_2(t)= 2\,t+3 $
and we choose 
\begin{equation*}
a_3(t)=\frac{t \,(t^2 + t + 1)}{13 \,t^2 +6 \,t+ 11}.
\end{equation*}
Then, we have a PH cubic given by
\begin{equation*}
\mathbf{r}(t)=\tfrac{1}{22\sqrt{13}} (5 + 6\, t - 9\, t^2,4+18\,t+6\,t^2,
 6 + 27\, t + 9\, t^2 + 13 \,t^3). 
\end{equation*}
This curve is depicted in Fig. \ref{fig:example_2}.

\begin{figure}
\begin{center}
\includegraphics[angle=0,width=6cm]{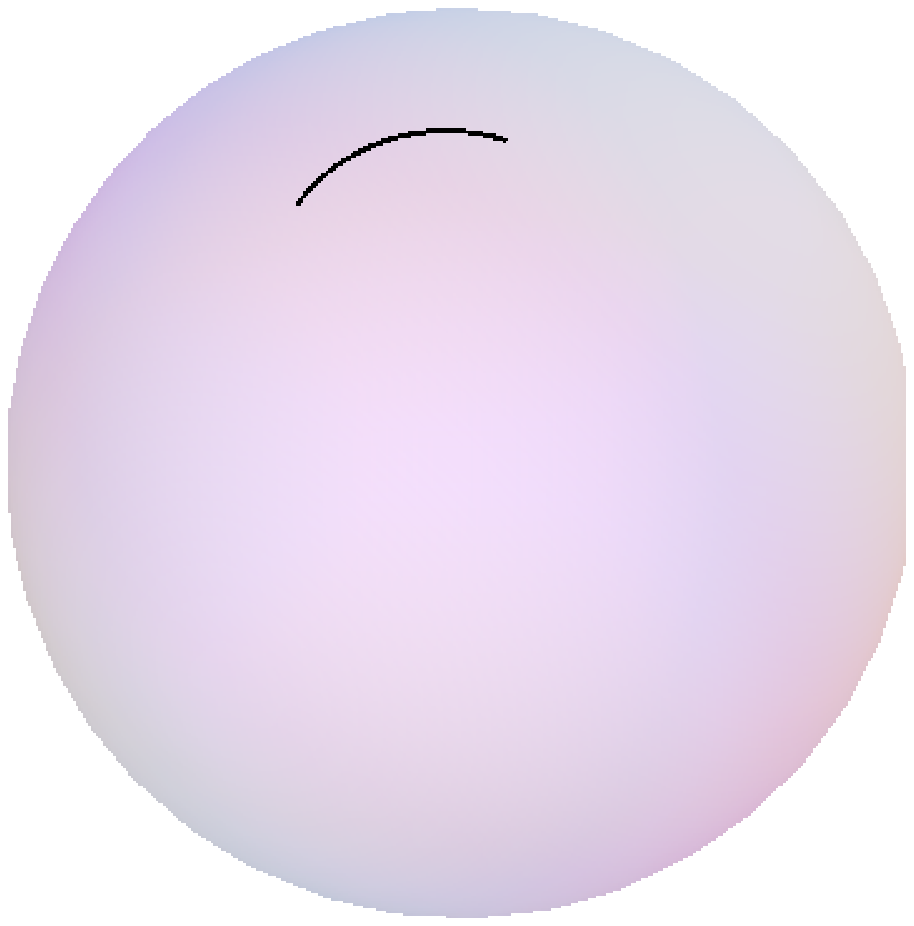}
\includegraphics[angle=0,width=6cm]{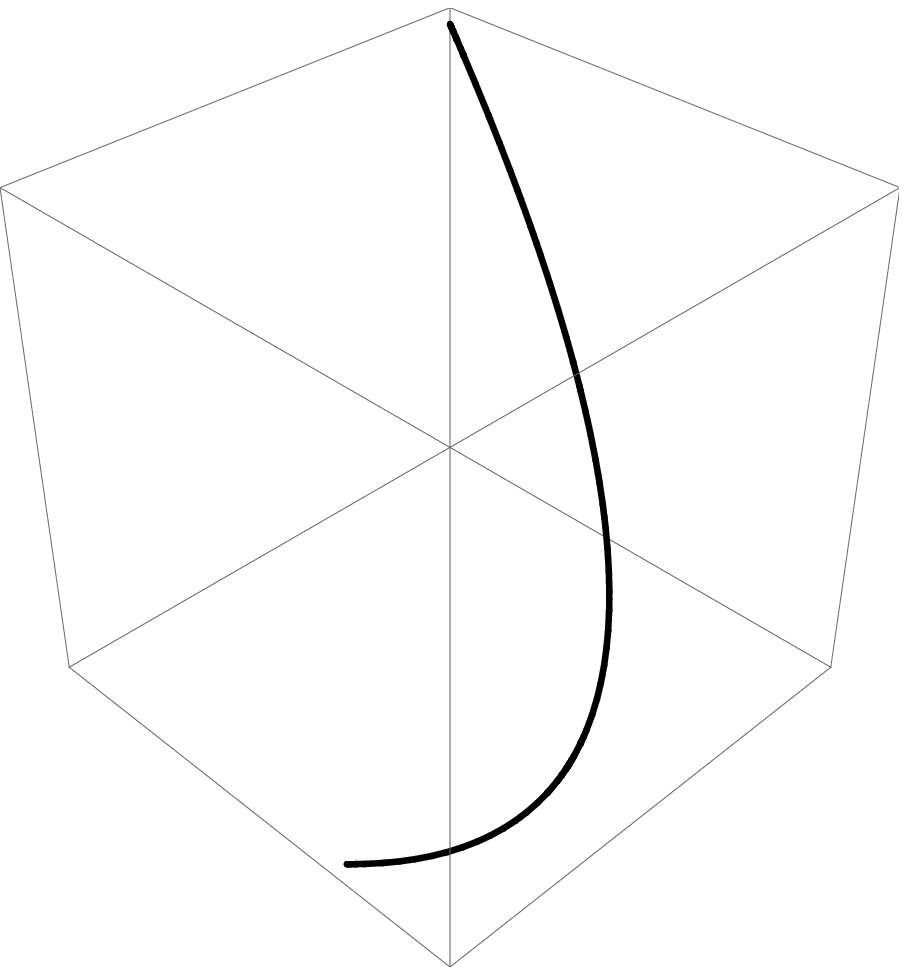}
\end{center}
\caption{\label{fig:example_2} Left: Tangent indicatrix of the PH cubic polynomial curve in
Example \ref{ex:PHcubic}. Right: The PH cubic polynomial curve ($0<t<1$).
}
\end{figure}
\end{example}

\section{$C^1$ Hermite interpolation}
\label{Hermite}
Rational PH curves are shown to be flexible enough for $C^1$ Hermite
interpolation in \cite{FaroukiSir11}. In this section, we outline that
the construction of rational helices given in the previous section
allows to interpolate given $C^1$ Hermite data 
\begin{equation}
\label{eq:HermiteData}
 \mathbf{r}(0)=p_0, \quad \mathbf{t}(0)=\mathbf{t}_0, \quad
\mathbf{r}(1)=p_1, \quad \mathbf{t}(1)=\mathbf{t}_1.
\end{equation}

Let us establish the rational helix $\mathbf{r}(t)$  with unit tangent $\mathbf{t}(t)$,
by interpolating these given Hermite data \eqref{eq:HermiteData}.
A way of finding a suitable $\mathbf{t}(t)$ is using the inverse of the
spherical projection \cite{FaroukiSir11}. First, the end tangents
$\mathbf{t}_0$ and $\mathbf{t}_1$ are projected on the plane by the spherical
projection, after finding the straight line or a circle
joining the two projection points, $\mathbf{t}(t)$
is obtained to be the inverse image of the straight line or the circle. Namely,
if the straight line or the circle is given by $(b_1(t),b_2(t))$, then
$\mathbf{t}(t)$ is given by \eqref{eq:Steographic_projection}.

Now, it remains to find the rational function $a_3(t)$.
This function $a_3(t)$ must have at least $6$ degrees of freedom,
in order to mach the end points $p_0$ and $p_1$ \cite{FaroukiSir11}.
Therefore, we choose it to be a rational Bezier curve of degree $3$
as  in \eqref{eq:rational_Bezier}. Then, we need $a_3(0),a_3'(0),a_3''(0),a_3(1),a_3'(1),a_3''(1)$
to determine the coefficients $c_0,c_1,c_2,c_3$
and the weights $w_1,w_2$.
 We claim that one can find 
$a_3(0),a_3'(0),a_3''(0),a_3(1),a_3'(1),a_3''(1)$ by the given data and
$\mathbf{t}(t)$. In fact, by \eqref{eq:rational_PH_helix} we have that
\begin{equation}\label{eq:formula_a_3}
 a_3(t)=\mathbf{r}(t) \cdot \mathbf{w}_3(t)=\frac{1}{\sin\psi}\,(\mathbf{r}(t)
\cdot \mathbf{u}-\cos\psi\, \mathbf{r}(t) \cdot \mathbf{t}(t)). 
\end{equation}
Consequently, $a_3(0),a_3(1)$ can be acquired. But taking derivatives of
\eqref{eq:formula_a_3} with respect to $t$ does not give us
$a_3'(0), a_3'(1)$. Therefore, use
\begin{equation}\label{eq:formula_a_2}
 a_2(t)=\mathbf{r}(t) \cdot \mathbf{w}_2(t)=\frac{1}{\sin\psi}
\det(\mathbf{t}(t),\mathbf{u},\mathbf{r}(t)),
\end{equation}
which can be obtained by \eqref{eq:rational_PH_helix}.
It is obvious that one can find $a_2(0), a_2'(0), a_2(1), a_2'(1)$
by making use of \eqref{eq:formula_a_2}, and
then \eqref{eq:second_orthogonality_result_helix}
gives us $a_3''(0), a_3''(1)$.

\begin{remark}
Recall that in the $C^1$ Hermite interpolation problem
for the function given in \eqref{eq:equation_function_f},
it is shown that $f(0), f'(0), f''(0), f(1), f'(1), f''(1)$ can be obtained \cite{FaroukiSir11}. 
Hence, we do not lose information by using an alternative method to construct
rational helical curves.  
\end{remark}

\begin{algorithm}
 Input: Hermite data ($p_0, \mathbf{t}_0,p_1, \mathbf{t}_1$)\\[1ex]
 1. Obtain tangent vector $ \mathbf{t}(t)$ by spherical projection. \\[1ex]
 2. Find $\mathbf{u}$ and $\cos \psi.$\\[1ex]
 3. Compute $a_3(0)$ and $a_3(1)$ from \eqref{eq:formula_a_3}.\\[1ex]
 4. Compute $a_3'(0)$ and $a_3'(1)$ from \eqref{eq:formula_a_2} and \eqref{eq:second_orthogonality_result_helix}.\\[1ex]
 5. Compute  $a_2'(0)$ and $a_2'(1)$ from derivative of equations  \eqref{eq:formula_a_2} and \eqref{eq:second_orthogonality_result_helix}.  \\[1ex]
 6. Computed $a_3''(0)$ and $a_3''(1)$ from equation \eqref{eq:second_orthogonality_result_helix}.\\[1ex]
 Output:  Rational helix curve \eqref{eq:rational_PH_helix}.
\end{algorithm}

\begin{example}
\label{ex:hermite}
We interpolate the $C^1$ Hermite data
given by 
\begin{equation*}
 p_0=(0,0,0), \, \mathbf{t}_0=(1,0,0), \, p_1 =(0, 1, -1),
\, \mathbf{t}_1=\frac{1}{3}(1, -2,1). 
\end{equation*}
We can interpolate the unit tangent $\mathbf{t}(t)$ by using
linear polynomials $b_1(t)=1, \, b_2(t)=-2 t.$
We find the rational Bezier curve \eqref{eq:rational_Bezier} of degree $3$ by 
\begin{equation*}
 c_0=c_1=c_2=0, \, c_3= -\frac{1}{3}, \, w_1= \frac{4}{9}, \, w_2=\frac{4}{9}.
\end{equation*}
This curve is depicted in Fig. \ref{fig:Hermite_1}.

\begin{figure}
\begin{center}
\includegraphics[angle=0,width=6cm]{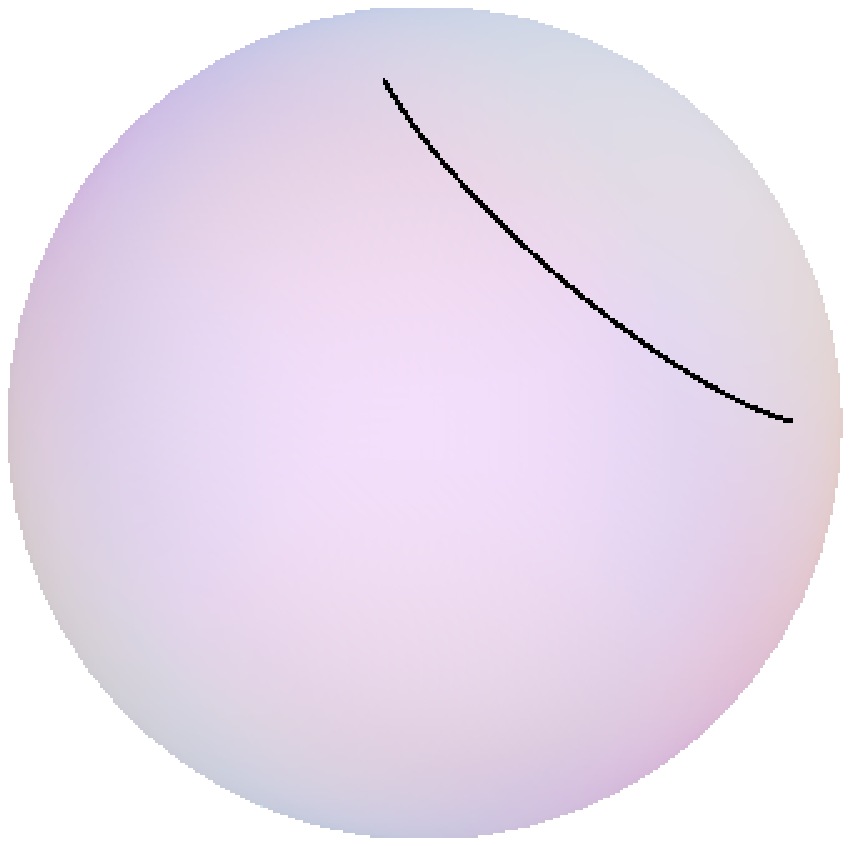}
\includegraphics[angle=0,width=6cm]{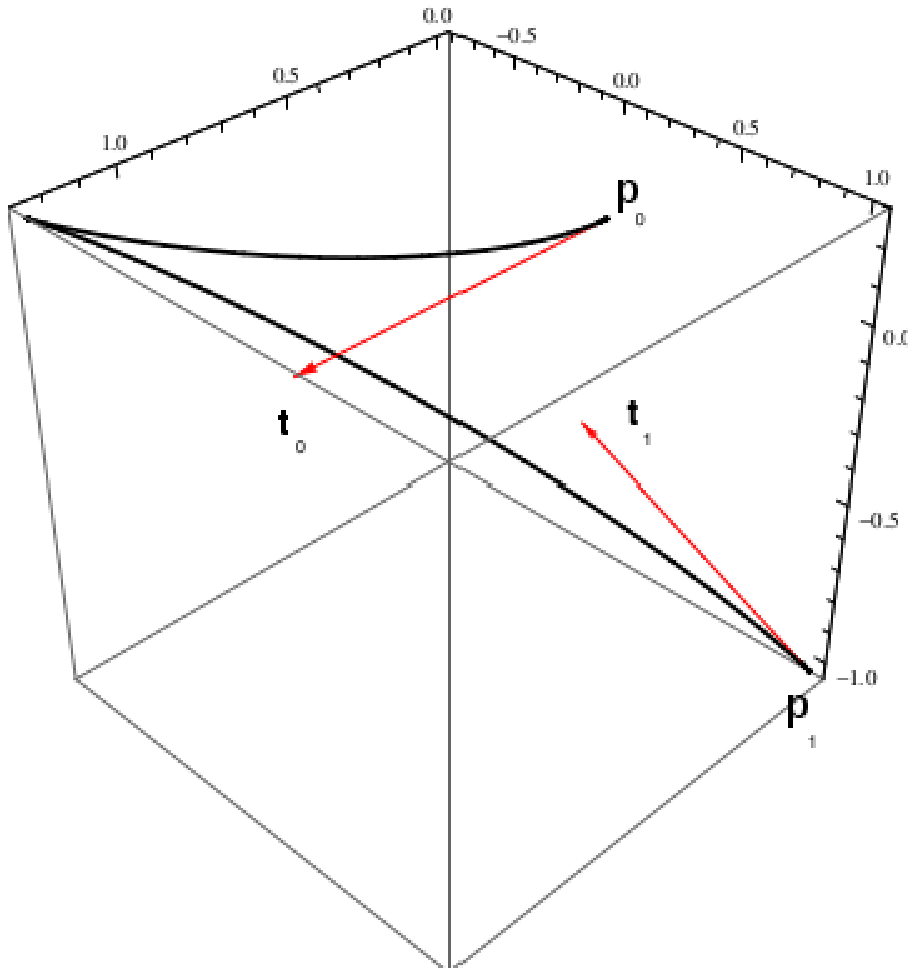}
\end{center}
\caption{\label{fig:Hermite_1} Left: The tangent indicatrix $\mathbf{t}(t)$ and (right) 
rational helical  $C^1$ Hermite interpolant for data $p_0=(0,0,0), \mathbf{t}_0=(1,0,0)$ and $p_1 =(0, 1, -1), \mathbf{t}_1=\frac{1}{3}(1, -2,1)$ (specified in Example \ref{ex:hermite}).
} 
\end{figure}
\end{example}
Note that we cannot claim that our interpolation method  is more efficient than the one given in \cite{FaroukiSir11}. But we wanted to demonstrate that  our method applies to 
that problem. It can be considered to remove the singularities as the next step.

\section{RRMFs on rational helices}
\label{sec:RRMF}
An \emph{adapted frame} along a space curve $\mathbf{r}(t)$
is an orthonormal basis $(\mathbf{f}_1(t),\mathbf{f}_2(t),\mathbf{f}_3(t))$ of
$\mathbb{R}^3$ such that $\mathbf{f}_1(t)=\mathbf{r}'(t)/|\mathbf{r}'(t)|$.
Recall that, among other equivalent definitions, the condition  
\begin{equation}\label{eq:rmfcon}
\mathbf{f}_2'(t)\cdot \mathbf{f}_3(t)\equiv0,
\end{equation}
is a necessary-and-sufficient condition for the frame to be rotation minimizing \cite{Han08}.
For computational purposes,  
it is desired that an adapted frame is
rational, such as a \emph{rational rotation
minimizing frame} (RRMF). However, not every PH curve admits an RRMF, it is shown by Han \cite{Han08} and  \c{S}eng\"{u}ler-\c{C}ift\c{c}i, Vegter   \cite{fatma-gert-2013} that
a cubic PH,  helical PH quintic curves do not admit an RRMF, respectively.

Now, let us assume that $\mathbf{r}(t)$ is a rational helix given by
\eqref{eq:rational_PH_helix}. 
In order to find the RRMF condition for rational helices, we utilize the same basis
$(\mathbf{w}_1(t),\mathbf{w}_2(t),\mathbf{w}_2(t))$ in \eqref{eq:helix_frame}.
Observe that an RMF is given by a rotation in the normal plane:
\begin{equation*}
 \mathbf{f}_1(t)=\mathbf{t}, \quad
 \begin{pmatrix}
  \mathbf{f}_2(t)  \\ \mathbf{f}_3(t) 
  \end{pmatrix}
  =
  \begin{pmatrix}
   \cos\theta(t) & -\sin \theta(t) \\   \sin \theta(t) &   \quad\cos\theta(t)
  \end{pmatrix}
 \begin{pmatrix}
\mathbf{w}_2(t)\\  \mathbf{w}_3(t)
  \end{pmatrix},
\end{equation*}
where 
\begin{equation}
\label{eq:theta}
 \theta(t) - \theta_0=-\int \tau \, |\mathbf{r}'(t)| \,dt,
\end{equation}
with the integration constant $\theta_0$ and the \emph{torsion} \eqref{eq:torsion} of the spine curve \cite{Farouki08}.
Therefore, an RMF
is not rational in general. As the frame $(\mathbf{w}_1(t),\mathbf{w}_2(t),\mathbf{w}_2(t))$
is rational, $(\mathbf{f}_1(t),\mathbf{f}_2(t),\mathbf{f}_3(t))$ is also rational
if and only if there exist rational functions
$\alpha, \beta, \gamma$ such that
\begin{equation*}
 \cos \theta(t) = \frac{\alpha}{\gamma} \quad \mbox{and} \quad \sin \theta(t) = \frac{\beta}{\gamma}.
\end{equation*}
Therefore, this is equivalent to the existence of relatively prime polynomials
$a(t)$ and $b(t)$ \cite{Han08}
satisfying
\begin{equation*}
 \cos \theta(t) = \frac{a^2(t)-a^2(t)}{a^2(t)+b^2(t)} \quad \mbox{and} \quad \sin \theta(t) = \frac{2\,a(t)\,b(t)}{a^2(t)+b^2(t)}.
\end{equation*}
Here observe that 
\begin{equation}\label{eq:theta_a_b}
\tan \frac{\theta(t)}{2}=\frac{a(t)}{b(t)}.  
\end{equation}

The following condition is gives a necessary and sufficient 
condition for a rational helix to have RRMFs \cite{fatma-gert-2013, Han08, FaroukiHan03}.  

\begin{proposition}\label{RRMF_condition}
Let a rational helix $\mathbf{r}(t)$ be given by (\ref{eq:rational_PH_helix}).
Then $\mathbf{r}(t)$ has an RRMF
if and only if there exist relatively prime polynomials $a(t)$ and $b(t)$ satisfying
\begin{equation}\label{eq:rrmf}
\frac{a(t) \,b'(t)-a'(t) \,b(t)}{a^2(t)+b^2(t)}=\frac{1}{2}\cot\psi\,|\mathbf{t}'(t)|.
\end{equation}
\end{proposition}
\begin{proof}
Obviously, by derivation of equation \eqref{eq:theta_a_b} 
\begin{equation*}
\frac{\theta'(t)}{2}=\frac{a(t)\,b'(t)-a'(t)\,b(t)}{a^2(t)+b^2(t)},  
\end{equation*}
and by equation \eqref{eq:theta}
\begin{equation*}
\theta'(t)=-\tau \, |\mathbf{r}'(t)|,
\end{equation*}
where $\kappa$ and $\tau$ the curvature and the torsion of the rational helix $\mathbf{r}(t)$.
By definition of curvature
\begin{equation*}
 \kappa = \frac{|\mathbf{t}'(t)|}{|\mathbf{r}'(t)|}.
\end{equation*}
Since $\mathbf{r}(t)$ is a helix, we have \cite{doCarmo76},
\begin{equation*}
 \frac{\kappa}{\tau} =\tan \psi.  
\end{equation*}
Thus the results follows.
\end{proof}

\begin{remark}\label{Remark_RRMF}
Proposition \ref{RRMF_condition} asserts that having an RRMF is
completely related to the unit tangent $\mathbf{t}(t)$. For helices,
the unit tangent is a small circle and therefore it can be obtained
by an inverse image of a circle or a straight line on the plane.
If the pole point is chosen as a point on the small circle, then
the image under the stereographic projection is a straight line.
Further choosing the coordinates properly, one can assume that
the image is a straight line parallel to one of the coordinate axes (see Fig. \ref{fig:chosing_projection}).
Working only with special choice of pole point and coordinate axes does not
affect the existence of nonexistence of RRMFs \cite{BartonJuttlerWang10}.
\end{remark} 

\begin{figure}
\begin{center}
\includegraphics[angle=0,width=8cm]{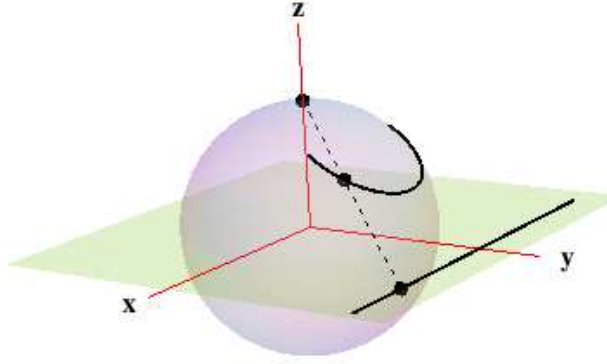}
\end{center}
\caption{\label{fig:chosing_projection}
For a suitable choice of the projection point and the coordinates, a small circle can be taken as the inverse image of a line parallel to the $x$-axis.
}
\end{figure}

\begin{theorem}\label{nonexistence_RRMF}
 Let $\mathbf{r}(t)$ be a rational helix with unit tangent
$\mathbf{t}(t)$ which is a rational vector field of degree $2$.
Then $\mathbf{r}(t)$ cannot have an RRMF.  
\end{theorem}
\begin{proof}
As we assume that  $\mathbf{t}(t)$ is a rational vector field of degree $2$,
as mentioned in Remark \ref{Remark_RRMF}, we can
choose $b_2(t)$  as a constant, say $n\in \mathbb{R}$, and  $b_1(t)=m \,t, \, m\in \mathbb{R}$ , in equation
\eqref{eq:Steographic_projection}.
In this case one obtains,
\begin{equation*}
\mathbf{u}=(1,0,n)/\sqrt{1+n^2},
\end{equation*}
since
\begin{equation*}
\mathbf{t}=(2n,2mt,-1+n^2+m^2t^2)/(1+n^2+m^2t^2).
\end{equation*}
Hence, we have, 
\begin{equation*}
\cot \psi=n. 
\end{equation*}
On the other hand, a simple computation gives that,
\begin{equation*}
|\mathbf{t}'|=2m/(1+n^2+m^2). 
\end{equation*}
Therefore \eqref{eq:rrmf} becomes
\begin{equation}\label{eq:RMFcon}
\frac{a(t) \,b'(t)-a' (t)\,b(t)}{a^2(t)+b^2(t)}=\frac{m\,n}{1+n^2(t)+m^2(t)}.
\end{equation}
Subsequently, the left hand side of \eqref{eq:RMFcon} for linear polynomials $a(t)= \, n_2\,t+n_3, \, b(t)=n_4\,t+n_5$ is
\begin{equation}\label{eq:right}
\frac{n_4 n_3- n_2 n_5}{(n_3 + n_2\, t)^2 + (n_5 + n_4 \,t)^2}.
\end{equation}
where $n_1, \,n_2,\,n_3,\,n_4, \,n_5\in \mathbb{R}.$ From the equality of \eqref{eq:RMFcon} and \eqref{eq:right}, we can get
the following equations
\begin{align*}
n_3 n_4 - n_2 n_5 - \lambda n_2 m=&0,\\
n_1^2 - n_2^2 - n_4^2 =&0,\\
  n_3^2 + n_5^2 - \lambda ( n^2 +1)=& 0,\\
 n_2 n_3 +  n_4 n_5 =&0,
\end{align*}
where $\lambda \in \mathbb{R}$.
An inspection shows that there do not exist $n_2, n_3, n_4, n_5,\lambda$ which satisfy
this set of equations.
Therefore, we demonstrated that there do not exist polynomials $a(t)$ and $b(t)$.
\end{proof}

There is not any RRMF on rational helices, for the simplest case where 
$\mathbf{t}(t)$ is a rational vector field of degree $2$.  Consequently, this motivates to develop a rational approximation to RMFs.

\section{Rational approximation on rational helices}
\label{sec:rational_approximation_RMF}

In this section, we 
will make a minimax rational approximation on rational helices by using
 \textsc{Mathematica}. A $(m,k)$ degree  rational function is the ratio of a degree 
$m$ polynomial to a degree $k$ polynomial. 
The \emph{error} of minimax rational approximation is the difference between the function 
and its approximation w.r.t. Euclidean norm.  The aim of minimax rational approximation is to 
minimize the maximum of the relative error from the polynomial curve. 

Let $h(t)$ be continuous on a closed interval $[t_0,t_1]$. Then there exists a unique $(m,k)$ 
degree rational polynomial $\frac{a(t)}{b(t)}$, called the \emph{minimax rational approximation} to exact function $h(t)$, that minimizes
\begin{equation*}
\varepsilon(a(t),b(t))=\max_{t_0<t<t_1} \mid h(t)-\tfrac{a(t)}{b(t)}\mid.
\end{equation*}

Nonexistence of RRMFs on rational helix curves
with tangent indicatrix of degree $2$ motivates an approximation of RRMFs which
 can be done as in \cite{FaroukiHan03} for PH polynomials, but in our case using the formula in Proposition
 \ref{RRMF_condition} which is special for rational helices of any degree.
Although, $\tan\frac{\theta(t)}{2}$ may not be rational, one can make
a rational approximation by 
\begin{equation}\label{tantetah}
\tan\frac{\theta(t)}{2}=-\tan \left(\int\,\frac{\tau\,\sigma}{2}  \,dt\right)\simeq\frac{a(t)}{b(t)},
\end{equation}
for some relatively prime polynomials
$a(t)$ and $b(t)$.
Utilizing this gives a rational frame 
\begin{equation}\label{eq:frame_approximation}
\tilde{\mathbf{f}}_1(t)=\mathbf{t}, \quad
 \begin{pmatrix}
  \tilde{\mathbf{f}}_2(t)  \\ \tilde{\mathbf{f}}_3(t) 
  \end{pmatrix}
  =
- \frac{1}{a^2(t)+b^2(t)}
  \begin{pmatrix}
    a^2(t)-b^2(t) & -2\,a(t)\,b(t) \\2\,a(t)\,b(t) &  a^2(t)-b^2(t)
  \end{pmatrix}
 \begin{pmatrix}
\mathbf{w}_2(t)\\  \mathbf{w}_3(t)
  \end{pmatrix},
\end{equation}
which is a rational approximation to the RMF. This approximation is done in the following example by using the 
minimax approximation procedure as explained before. 

\begin{example}\label{ex:RRRM_approximation}
In this example $\mathbf{t}(t)$ is chosen by \eqref{eq:Steographic_projection}
where $b_1(t)=t$ and  $b_2(t)= t-1 $. Then $a_3(t)$ is chosen as a rational Bezier curve of degree $3$ given by
\eqref{eq:rational_Bezier} with
\begin{equation*}
 c_0 = 1, \, c_1 = 2, \, c_2 = 0, \, c_3 = 0, \, w_1 = \frac{1}{2}, \, w_2 = 1.
\end{equation*}
Hence, we obtain a rational helix of degree $9$. 
Then we compute 
\begin{equation*}
\tan \frac{\theta(t)}{2}=\frac{1}{2 - 2 \,t + 2 \,t^2}.
\end{equation*}
A minimax rational approximation of degree $(3,3)$ to that function
is obtained to be
\begin{equation*}
\tan \frac{\theta(t)}{2}\simeq \frac{0.188141 + 0.445412\, t - 0.0170917\, t^2 + 0.15731 \,t^3}{1 - 
 0.549016 \,t + 0.513789 \,t^2 - 0.011689 \,t^3}
 \end{equation*}
with error $3.63871 \times 10^{-6}$.
With $a(t)$ and $b(t)$ at hand one can compute a rational
approximation to the RMF by \eqref{eq:frame_approximation}, which is depicted in Fig.
\ref{fig:rotations}. The frame is not a RMF but it is close to
satisfy RMF condition, the RMF condition error is given in Fig. \ref{fig:RRMF_error}. 
\begin{figure}
\begin{center}
\includegraphics[angle=0,width=6cm]{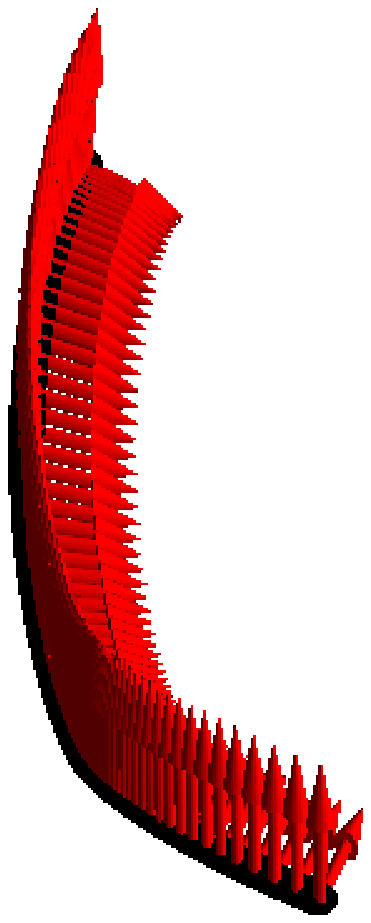}
\includegraphics[angle=0,width=6cm]{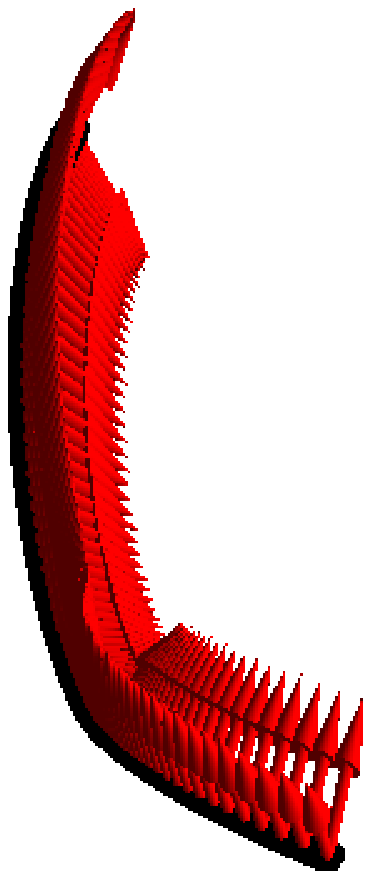}
\end{center}
\caption{\label{fig:rotations} Comparison of rational approximation to (left) RMF and  (right) FSF on a rational helix (for clarity, the tangent vector is omitted).
}
\end{figure}

\begin{figure}
\begin{center}
\includegraphics[angle=0,width=6cm]{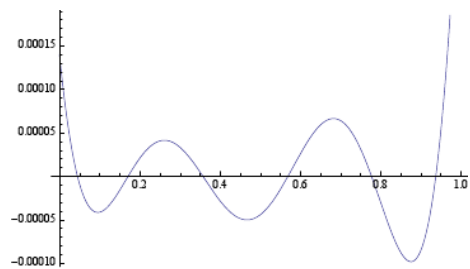}
\end{center}
\caption{\label{fig:RRMF_error} Error for the RMF condition \eqref{eq:rmfcon}.
}
\end{figure}
\end{example}

\section{Rational approximation of profile surfaces}
\label{sec:surface}
A \emph{profile surface} is a sweep surface generated
by an RMF. More explicitly,
it has a parametric representation
\begin{equation*}
\textbf{S}(s,t)=\,\mathbf{r}(t)+\mathbf{f}_2(t)\,c_1(s)+\mathbf{f}_3(t)\, c_2(s),
\end{equation*}
where $\mathbf{r}(t)$ is the
\emph{spine curve} with parameter $t\in[t_0,t_1]\in\mathbb{R},$
$c(s)=(c_1(s),c_2(s))^T$ is the  \emph{cross section} or \emph{profile curve}
with parameter $s\in[s_0,s_1]\subset \mathbb{R}$, and
$(\mathbf{f}_1(t),\mathbf{f}_2(t),\mathbf{f}_3(t))$ is an RMF
 along $\mathbf{r}(t)$.

If the cross section curve is a straight line, then the profile surface is a developable surface \cite{Maurer99}. This implies that they are flat surfaces, i.e.
they have vanishing Gauss curvature $K=0.$
In the next section, we obtain a rational approximation
of an RMF on rational helices and we derive profile surfaces with this rational helical curve.

Rational approximation of RMF can be used to
generate rational approximations to profile surfaces. If the profile curve $c(s)$
is chosen to be a straight line then the rational approximation to
the profile surface is expected to 
have Gauss curvature close to zero values.
\noindent
\begin{example}
Consider two sweep surfaces, 
\begin{equation*}
\begin{split}
 \textbf{S}_1(s,t)=\,&\mathbf{r}(t)+(-\tfrac{1}{5} \,s + 5)\, \mathbf{\tilde{f}}_2(t) + (10\, s - \tfrac{1}{2})\, \mathbf{\tilde{f}}_3(t),\\[0.6ex]
 \textbf{S}_2(s,t)=\,&\mathbf{r}(t)+(-\tfrac{1}{5} \,s + 5)\, \frac{1}{\sin\psi} \mathbf{u}\times\mathbf{v}(t) + (10 \,s -  \tfrac{1}{2})\, \mathbf{w}_3(t),
\end{split}
\end{equation*}
generated by the rational approximation to the RMF (left) and by the FSF (right)
of the rational helical given in Example \ref{ex:RRRM_approximation}, see Figure \ref{fig:sweep_surfaces}.

The Gaussian curvature $\tilde{K}$ can be used as an accuracy criterion. 
Since the cross section curve 
\begin{equation*}
c(s)=\,(-\tfrac{1}{5} \,s + 5,10\, s - \tfrac{1}{2})^T
\end{equation*}
in this example is a straight line,
the Gauss curvature of a profile surface is vanishing.  For  profile surface $\textbf{S}_1(s,t),$
 minimum and maximum  values of the Gauss curvature are
\begin{equation*}
\begin{split}
\tilde{K}_{min}( 0.899997, 1)=&-6.87389\times10^{-20}, \\[0.6ex]
\tilde{K}_{max}(0.899991, 4.82357)=&-6.02541\times10^{-13}.
\end{split}
\end{equation*} 
Our approximation $\tilde{K}$ is between 
the values $\tilde{K}_{min}$ and $\tilde{K}_{max}$ which are close to zero.
Therefore  this criterion shows us that our approximation gives hight precision results.
\end{example}

\begin{figure}
\begin{center}
\includegraphics[angle=0,width=7cm]{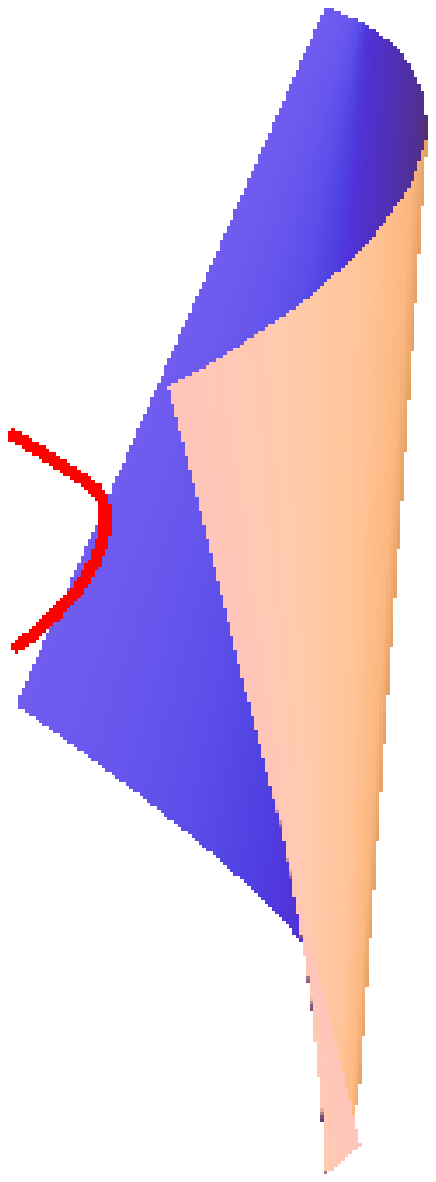}
\includegraphics[angle=0,width=7cm]{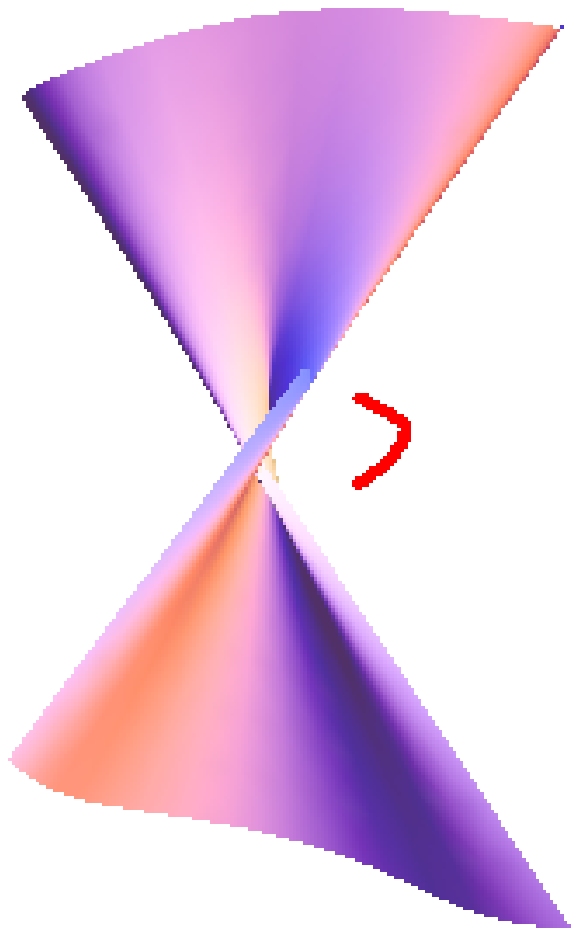}
\end{center}
\caption{\label{fig:sweep_surfaces} Profile surfaces $S_1(s,t)$ and $S_2(s,t),$ generated by rational approximation to the RMF (left) and by the FSF (right).
}
\end{figure}

\section{Concluding remarks}
\label{concl}
Rational representation of PH curves and rational approximation to RMFs are two main topics for 
computer graphics, swept surface
or generalized cylinder constructions, motion design and control in computer animation and streamline visualization in CAD/CAM.  The importance reasons are: RMFs have property of minimum twist which makes them useful and rational expressions simplify the calculations. In the paper, we have discussed the geometry of rational helices and we
applied it to RRMFs. 

Although polynomial helices are well-documented, rational helices
have not investigated enough. We hope our study sheds light on
future work. There arises many related future problems:
\begin{itemize}
\item Removing the singularities (cusp points) of rational helices. 
\item When rational helices are polynomial curves?
\item Rigid body design is another application which deserves more attention.
\end{itemize}


\section*{Acknowledgement} The author was supported by the Dutch National Science Foundation (NWO) under Grant  435053, the project name is Certified Geometric Approximation (CGA). 
We would like to thank Prof.Dr. Arthur E. P. Veldman for helpful discussions.


\bibliographystyle{unsrt}
\bibliography{bib}

\end{document}